\documentclass[envcountsect,oribibl,orivec]{llncs}
\usepackage{amssymb}
\usepackage{amsmath}
\usepackage[right=4cm,left=4cm]{geometry}

\usepackage{bussproofs}

\EnableBpAbbreviations
\usepackage{graphics}
 \usepackage{graphicx}

  \usepackage{qtree}
\usepackage{tikz}

\makeatletter
\newcommand{\myLines}[1]{
\begin{picture}(4,1)
\put(0,0){\line(2,1){2}}
\put(2,0){\vector(0,1){0.7}}
\put(4,0){\line(-2,1){2}}
\end{picture}}
\let\qdrawReal=\qdraw@branches
\newcommand\brOverride{\let\qdraw@branches=\myLines}
\newcommand\brRestore{\let\qdraw@branches=\qdrawReal}
\makeatother
\newcommand{\JL}{{\sf JL}}
\newcommand{\LP}{{\sf LP}}
\newcommand{\M}{{\mathcal M}}
\newcommand{\E}{{\mathcal E}}

\newcommand{\V}{{\mathcal V}}
\newcommand{\CS}{\mathcal{CS}}

 \def\r{\rightarrow}

\begin{document}
\title{Analytic Tableaux for Justification Logics}

\author{Meghdad Ghari}
\institute{School of Mathematics,
Institute for Research in Fundamental Sciences (IPM), \\ P.O.Box: 19395-5746, Tehran, Iran \\ \email{ghari@ipm.ir}
}

\maketitle
\begin{abstract}
In this paper we present  analytic tableau proof systems for various justification logics. We show that the tableau systems are sound and complete with respect to Mkrtychev models. In order to prove the completeness of the tableaux, we give a syntactic proof of cut elimination. We also show the subformula property for our tableaux, and prove the decidability of justification logics for finite constant specifications.  \\

\end{abstract}

\section{Introduction}

Justification logics are modal-like logics that provide a framework for reasoning about epistemic justifications (see \cite{A2008,ArtemovFitting,Fitting2008} for a survey). The language of justification logics extends the language of propositional logic by justification terms and expressions of the form $t:A$, with the intended meaning ``$t$ is a justification for $A$''.  The first logic in the family of justification logics, \emph{the Logic of Proofs} \LP, was introduced by  Artemov in \cite{A1995,A2001}. The logic of proofs is a counterpart of modal logic {\sf S4}. Other logics of this kind have been introduced so far (cf. \cite{KuznetsStuder2012}).  In this paper we deal only with those justification logics which are counterparts of normal modal logics between {\sf K} and {\sf S5}.

Various tableau proof systems have been developed for the logic of proofs (see \cite{Finger2010,Fitting2005,Ghari-tableaux-2016,Renne2004,Renne2006}). However, it seems that the only
analytic tableau proof system is Finger's prefixed KE tableaux  \cite{Finger2010}. Tableau proof systems for other justification logics can be found in \cite{Ghari-tableaux-2016}.

  The aim of this paper is to present analytic tableau proof systems for justification logics.  All tableau proof systems are sound and complete with respect to Mkrtychev models of justification logics. In order to prove the completeness of the tableaux, we give a syntactic proof of cut elimination. We  show that our tableau systems enjoy the subformula property, and we prove the decidability of justification logics for finite constant specifications.
 
\section{Justification logics}\label{sec:Justification logics}
The language of justification logics is an
extension of the language of propositional logic by the formulas
of the form $t:F$, where $F$ is a formula and $t$ is a
justification term. \textit{Justification terms} (or
\textit{terms} for short) are built up from (justification)
variables $x, y, z, \ldots$ and (justification) constants $a,b,c,\ldots$  using several operations depending on the logic: (binary) application `$\cdot$', (binary) sum `$+$', (unary) verifier `$!$', (unary) negative verifier `$?$', and (unary) weak negative verifier `$\bar{?}$'. Subterms of a term are defined in the usual way: $s$ is a subterm of $s,
s+t, t+s, s\cdot t$, $!s$, $\bar{?}s$, and  $?s$.

$\JL$-formulas are constructed from a countable set of propositional variables, denoted $\mathcal{P}$, by the following grammar:
\[ A::= p~|~\bot~|~\neg A~|~A\rightarrow A~|~t:A,\]
where $p\in\mathcal{P}$ and $t$ is a justification term. Other Boolean connectives are defined as usual.

For a $\JL$-formula $A$, the
set of all subformulas of $A$, denoted by $Sub(A)$, is defined inductively as follows:
$ Sub(p)=\{p\}$ where $p\in\mathcal{P}$;
$Sub(\bot)=\{\bot\}$; $Sub(A\r B)=\{A\r B\}\cup Sub(A)\cup
Sub(B)$; $Sub(t:A)=\{t:A\}\cup Sub(A)$. For a set $S$ of \JL-formulas, $Sub(S)$ denotes the set of all subformulas of the formulas from $S$.

We now begin with describing the axiom schemes and rules of the basic
justification logic {\sf J}, and continue with other justification
logics. The basic justification logic {\sf J} is the weakest
justification logic we shall be discussing. Other
justification logics are obtained by adding certain axiom schemes
to {\sf J}.
\begin{definition}\label{def: justification logics}
Axioms schemes of {\sf J} are:
\begin{description}
\item[Taut.] All propositional tautologies,
\item[Sum.] $s:A\rightarrow (s+t):A~,~s:A\rightarrow (t+s):A$,
\item[jK.] $s:(A\rightarrow B)\rightarrow(t:A\rightarrow (s\cdot t):B)$.
\end{description}
Other justification logics are obtained by adding the following axiom schemes to {\sf J} in various combinations:
\begin{description}
\item[jT.] $t:A\rightarrow A$.
\item[jD.] $t:\perp \rightarrow \perp$.
\item[j4.] $t:A\rightarrow !t:t:A$,
\item[jB.] $\neg A\rightarrow\bar{?} t:\neg t: A$.
\item[j5.] $\neg t:A\rightarrow ?t:\neg t:A$.
\end{description}
All justification logics have the inference rule Modus Ponens, and the \textit{Iterated Axiom Necessitation} rule:
\begin{description}
\item[IAN.]
$\vdash c_{i_n}:c_{i_{n-1}}:\ldots:c_{i_1}:A$, where $A$ is an axiom instance of the logic, $c_{i_j}$'s
are arbitrary justification constants and $n\geq 1$.
\end{description}
\end{definition}
In what follows, {\sf JL} denotes any of the justification logics defined in Definition \ref{def: justification logics}, unless stated otherwise. The language of each justification logic $\JL$ includes those operations on terms that are present in its axioms. $Tm_\JL$ and $Fm_\JL$ denote the set of all terms and the set of all formulas of $\JL$ respectively. Moreover, the name of each justification logic is indicated by the list of its axioms. For example, ${\sf JT4}$ is the extension of ${\sf J}$ by axioms jT and j4, in the language containing term operations $\cdot$, $+$, and $!$. {\sf JT4} is usually called the logic of proofs $\LP$.

\begin{definition}
A \textit{constant specification} $\CS$
for \JL~is a set of formulas of the form
$c_{i_n}:c_{i_{n-1}}:\ldots:c_{i_1}:A$, where $n\geq 1$, $c_{i_j}$'s are
justification constants and $A$ is an axiom instance of \JL, such that it is downward closed: if $c_{i_n}:c_{i_{n-1}}:\ldots:c_{i_1}:A\in\CS$, then $c_{i_{n-1}}:\ldots:c_{i_1}:A\in\CS$.
\end{definition}

The typical form of a formula in a constant specification for \JL~is $c:F$, where $c$ is a justification constant, and $F$ is either an axiom instance of \JL~or of the form $c_{i_m}:c_{i_{m-1}}:\ldots:c_{i_1}:A$, where $m\geq 1$, $c_{i_j}$'s are justification constants and $A$ is an axiom instance of \JL.

Let ${\sf JL}_\CS$ be the fragment of ${\sf JL}$ where the Iterated Axiom Necessitation rule only produces formulas from the given $\CS$.

In the remaining of this section, we recall the definitions of M-models for justification logics (see \cite{Mkrtychev1997, KuznetsStuder2012}).

\begin{definition}\label{Kripke-Fitting models J}
 An M-model $\M=(\E, \V)$ for justification
logic ${\sf J}_\CS$ consists of a valuation $\V:\mathcal{P} \rightarrow \{0,1\}$ and an admissible evidence function $\E:Tm_\JL \rightarrow 2^{Fm_\JL}$ meeting the following conditions:
\begin{description}
 \item[$\E 1.$]  $A\r B\in\E(s)$ and $A\in\E(t)$ implies $B\in\E(s\cdot t)$.
  \item[$\E 2.$]  $\E(s)\cup \E(t)\subseteq\E(s+t)$.
  \item[$\E 3.$]  $c:F\in\CS$ implies $F\in\E(c)$.
  \end{description}
  \end{definition}

Informally, $A \in \E(t)$ means ``term $t$ is an admissible evidence for formula $A$."

\begin{definition}\label{def:forcing relation}
For an M-model $\M=(\E, \V)$ the forcing relation $\Vdash$ is defined as follows:
\begin{enumerate}
\item $\M\not\Vdash \bot$,
\item $\M\Vdash p$ if{f} $\V(p)=1$, for $p\in\mathcal{P}$,
 \item $\M\Vdash \neg A$ if{f} $\M\not\Vdash A$,
  \item $\M\Vdash A\r B$ if{f} $\M\not\Vdash A$ or $\M\Vdash B$,
 \item  $\M\Vdash t:A$ if{f} $A\in\E(t)$.
\end{enumerate}
If $\M\Vdash F$ then it is said that $F$ is true in $\M$ or $\M$ satisfies $F$.
\end{definition}

In order to define M-models for other justification logics of Definition \ref{def: justification logics} certain
additional conditions should be imposed on the M-model.

\begin{definition}\label{Kripke-Fitting models JL}
 An M-model $\M=(\E, \V)$ for justification logic ${\sf JL}_\CS$ is an M-model for ${\sf J}_\CS$ such that:
\begin{itemize}
\item if $\JL$ contains axiom jT, then for all $t\in Tm_\JL$ and $A\in Fm_\JL$:
\begin{description}
   \item[$\E 4.$]  $A\in\E(t)$ implies $\M \Vdash A$.\vspace{0.05cm}
     \end{description}
\item if $\JL$ contains axiom jD, then for all $t\in Tm_\JL$:
\begin{description}
   \item[$\E 5.$]  $\bot\not\in\E(t)$.\vspace{0.05cm}
     \end{description}
\item if $\JL$ contains axiom j4, then for all $t\in Tm_\JL$ and $A\in Fm_\JL$:
\begin{description}
   \item[$\E 6.$]  $A\in\E(t)$ implies $t:A\in\E(!t)$.\vspace{0.05cm}
     \end{description}
\item if $\JL$ contains axiom jB, then for all $t\in Tm_\JL$ and $A\in Fm_\JL$:
\begin{description}
     \item[$\E 7.$]  $\M\not\Vdash A$ implies $\neg t:A\in\E(\bar{?}t)$.\vspace{0.05cm}
    \end{description}
\item if $\JL$ contains axiom j5, then for all $t\in Tm_\JL$ and $A\in Fm_\JL$:
\begin{description}
       \item [$\E 8.$] $A\not\in\E(t)$ implies $\neg t:A\in\E(?t)$.
         \end{description}
\end{itemize}
\end{definition}

By a $\JL_\CS$-model we mean an M-model for justification logic $\JL_\CS$. A \JL-formula $F$ is $\JL_\CS$-valid if it is true in every $\JL_\CS$-model. For a set $S$ of formulas, $\M\Vdash S$ provided that $\M\Vdash F$ for all formulas $F$ in $S$. Note that given a constant specification $\CS$ for \JL, and a model $\M$ of $\JL_\CS$ we have $\M\Vdash \CS$ (in this case it is said that $\M$ respects $\CS$).

The proof of soundness and completeness theorems for all justification logics of Definition \ref{def: justification logics} are given in \cite{KuznetsPhD2008,KuznetsStuder2012}.

\begin{theorem}\label{thm:Sound Compl JL}
Let \JL~be one of the justification logics of Definition \ref{def: justification logics}, and $\CS$ be a constant specification for \JL. Then \JL-formula $F$ is provable in  $\JL_\CS$ if{f} $F$ is $\JL_\CS$-valid.
\end{theorem}

\section{Analytic tableaux}

In this section we present analytic tableaux for justification logics.  A restricted form of the cut rule, called the principle of bivalence in \cite{D’Agostino1992,D’Agostino1999,D’AgostinoMondadori1994} and denoted by $(PB)$, is used in the tableaux. In order to make the cut rule  analytic we put a restriction on its applications. Let us first extend the definition of subformulas of a  formula to include constant specifications.

\begin{definition}\label{def:CS-subformula}
Given a constant specification $\CS$, a formula $A$ is $\CS$-subformula of a formula $B$ if  $A$ is a subformula of a formula in $\{B\} \cup \CS$; i.e. $A \in Sub(\{B\} \cup \CS)$. 

 A formula $A$ is a weak $\CS$-subformula of a formula $B$ if  $A$ is either a subformula of a formula in $\{B\} \cup \CS$ or the negation of a subformula of a formula in $\{B\} \cup \CS$; i.e. $A \in Sub(\{B\} \cup \CS)$ or $A=\neg C$ and $C \in Sub(\{B\} \cup \CS)$.
\end{definition}

Tableau proofs will be of $\JL$-formulas but in addition to $\JL$-formulas will use formulas of the form $[t,A]$ or $\sim [t,A]$, where $A$ is a $\JL$-formula and $t$ is a term. The formulas $[t,A]$ and $\sim [t,A]$ are called  \textit{evidential formulas}. The intended meaning of $[t,A]$ is ``$t$ is an admissible evidence for $A$", and $\sim [t,A]$ is intended to be the negation of $[t,A]$.

A  ${\sf J}_\CS^e$-tableau for a $\JL$-formula is a binary tree with the negation of  that formula at the root constructed by applying  tableau rules from Table \ref{table:E-J}. For extensions of {\sf J}, tableau rules corresponding to axioms from Table \ref{table:E-tableau rules JL} should be added to ${\sf J}_\CS^e$-tableau rules. For example, the tableau proof system of the logic of proofs $\LP$ is obtained by adding the rules $(e)$ and $(!)$ to ${\sf J}_\CS^e$-tableau rules. 

For a justification logic $\JL_\CS$, a $\JL_\CS^e$-tableau branch closes if one of the following holds:

\begin{enumerate}
\item Both $A$ and $\neg A$ occurs in the branch, for some $\JL$-formula $A$.
\item Both $[t,A]$ and $\sim [t,A]$ occurs in the branch, for some term $t$ and $\JL$-formula $A$.
\item $\bot$ occurs  in the branch.
\item $\neg c:F$ occurs in the branch, for some $c:F\in\CS$.
\end{enumerate}

A $\JL_\CS^e$-tableau closes if all branches of the tableau close. A $\JL_\CS^e$-tableau proof for a formula $F$ is a closed tableau beginning with $\neg F$ (the root of the tableau) using only $\JL_\CS^e$-tableau rules. A $\JL_\CS^e$-tableau for a finite set $S$ of $\JL$-formulas begins with a single branch whose nodes consist of the formulas of $S$ as roots.

Note that in $\JL_\CS^e$-tableaux the rules $(\cdot)$, $(PB)$,  and $(PB_e)$ have restrictions on their applications (see Table \ref{table:E-J}). The formula $A$ in the conclusion of $(PB)$ is called the $PB$-formula. Furthermore, the rule $(\cdot)$ is a binary rule (it takes two formulas as input), and it should be read as follows: if a branch contains $s:(A\rightarrow B)$ and $t:A$, then we can extend that branch by adding $s\cdot t:B$, provided that the formula $A\rightarrow B$ is a $\CS$-subformula of the root of the tableau and $s \cdot t$ occurs in the root. In addition, there is no ordering intended on the inputs $s:(A\rightarrow B)$, $t:A$.

From Definition \ref{def:CS-subformula} it is obvious that the following is an instance of $(PB)$: 

\begin{prooftree}
\AXC{}\RightLabel{$(PB)$}
\UIC{$c_{i_n}:c_{i_{n-1}}:\ldots:c_{i_1}:A~|~\neg c_{i_n}:c_{i_{n-1}}:\ldots:c_{i_1}:A$}
\end{prooftree}

where $c_{i_n}:c_{i_{n-1}}:\ldots:c_{i_1}:A\in\CS$. 
Since the right fork is closed, it follows that the following rule is admissible in $\JL^e_\CS$: 

\begin{prooftree}
\AXC{}
\UIC{$c_{i_n}:c_{i_{n-1}}:\ldots:c_{i_1}:A$}
\end{prooftree}
where $c_{i_n}:c_{i_{n-1}}:\ldots:c_{i_1}:A\in\CS$.

\begin{example}
We give a ${\sf J}_\CS^e$-tableau proof of $x:A \r c \cdot x:(B\r A)$, where $\CS$ contains $c:(A\r (B\r A))$.

\vspace*{0.2cm}
\Tree [.$1.~\neg(x:A\r c\cdot x:(B\r A))$ [.$2.~x:A$ [.$3.~\neg c\cdot x:(B\r A)$ [.$4.~[x,A]$ [.$5.~\sim[c\cdot x,B\r A]$ [.$6.~c:(A\r (B\r A))$ [.$8.~[c,A\r (B\r A)]$ { $9.~[c\cdot x,B\r A]$  \\ $\otimes$} ] ] !\qsetw{5cm} {$7.~\neg c:(A\r (B\r A))$ \\ $\otimes$} ] ] ] ] ]
\vspace*{0.2cm}

Formulas 2 and 3 are from 1 by rule $(F\r)$, 4 is from 2 by rule $(Te)$, 5 is from 3 by rule $(Fe)$, 6 and 7 are obtained by $(PB)$, 8 from 6 by $(Te)$, and 9 from 2 and 8 by rule $(\cdot)$. Note that in the application of $(PB)$ the $PB$-formula $c:(A\r (B\r A))$ is a $\CS$-subformula of the root, and in the application of $(\cdot)$ the formula $A\r (B\r A)$ is a $\CS$-subformula of the root and $c \cdot x$ occurs in the root.
\end{example}

\begin{table} 
\centering\renewcommand{\arraystretch}{1.5}
\begin{tabular}{|cc|}
\hline
\multicolumn{2}{|c|}{
\AXC{}\noLine\UIC{$\neg\neg A$}\RightLabel{$(F\neg)$}
\UIC{$A$}\noLine
\UIC{}
\DP}
\\
\AXC{}\noLine 
\UIC{$\neg(A\rightarrow B)$}\RightLabel{$(F\r)$}
\UIC{$A$}\noLine
\UIC{$\neg B$}\noLine
\UIC{}
\DisplayProof
&
\AXC{$A\rightarrow B$}\RightLabel{$(T\r)$}
\UIC{$\neg A | B$}
\DP
\\\hline

 \AXC{}\noLine\UIC{$t:A$} \RightLabel{$(Te)$}
 \UIC{$[t,A]$}\noLine
\UIC{}
 \DP
 &
\AXC{}\noLine\UIC{$\neg t:A$} \RightLabel{$(Fe)$}
 \UIC{$ \sim [t,A]$}\noLine
\UIC{}
 \DP
 
\\\hline
\AXC{}\noLine\UIC{$\sim [t+s,A]$}\RightLabel{$(+_L)$}
\UIC{$\sim [t,A]$}
\DisplayProof
&
\AXC{}\noLine\UIC{$\sim [t+s,A]$}\RightLabel{$(+_R)$}
\UIC{$\sim [s,A]$}
\DisplayProof
\\
\multicolumn{2}{|c|}{
\AXC{}\noLine
\UIC{}\noLine
\UIC{$[s,(A\rightarrow B)]$}\noLine
\UIC{$[t,A]$}\RightLabel{$(\cdot)$}
\UIC{$[s\cdot t,B]$}\noLine
\UIC{}
\DisplayProof
}
\\\hline
\AXC{}\noLine 
\UIC{}\noLine 
\UIC{} 
\RightLabel{$(PB)$}
\UIC{$A~|~\neg A$}\noLine
\UIC{}
\DisplayProof 
&
\AXC{}\noLine 
\UIC{}\noLine 
\UIC{}  
\RightLabel{$(PB_e)$}
\UIC{$[t,A]~|~\sim [t,A]$}\noLine
\UIC{}
\DisplayProof
\\\hline
\multicolumn{2}{|l|}{In $(\cdot)$ the formula $A \r B$ is a $\CS$-subformula of the root and}\\\multicolumn{2}{|l|}{ the term $s \cdot t$ occurs in the root.}\\\hline
\multicolumn{2}{|l|}{In $(PB)$ the $PB$-formula $A$ is a $\CS$-subformula of the root.}\\\hline
\multicolumn{2}{|l|}{In $(PB_e)$ the formula $A$ is a $\CS$-subformula of the root and}\\\multicolumn{2}{|l|}{ the term $t$ occurs in the root.}\\\hline
\end{tabular}\vspace{0.3cm}
\caption{Tableau rules for basic justification logic {\sf J}.}\label{table:E-J}
\end{table}

\begin{table}
\centering\renewcommand{\arraystretch}{2}
\begin{tabular}{|l|c|}
\hline
 ~Justification axiom & Tableau rule~ \\\hline
 ~{\bf jT}. $t:A\r A$ &
 \AXC{}\noLine 
\UIC{$[t,A]$} \RightLabel{$(e)$}
 \UIC{$A$}\noLine
\UIC{}
 \DP \\\hline
 ~{\bf jD}. $t:\bot\r\bot$  &
 \AXC{}\noLine 
\UIC{$[t,\bot]$} \RightLabel{$(e_\bot)$}
 \UIC{$ \bot$}\noLine
\UIC{}
 \DP
  \\\hline
 ~{\bf j4}. $t:A\r !t:t:A$  &
 \AXC{}\noLine 
\UIC{$\sim [!t,t:A]$} \RightLabel{$(!)$}
 \UIC{$\sim [t,A]$}\noLine
\UIC{}
 \DP
 \\\hline
 ~{\bf jB}. $\neg A\r \bar{?}t:\neg t:A$ &
 \AXC{}\noLine 
\UIC{$\sim [\bar{?}t,\neg t:A]$} \RightLabel{$(\bar{?})$}
 \UIC{$A$}\noLine
\UIC{}
 \DP
  \\\hline
 ~{\bf j5}. $\neg t:A\r ?t:\neg t:A$ &
 \AXC{}\noLine 
\UIC{$\sim [?t,\neg t:A]$} \RightLabel{$(?)$}
 \UIC{$[t,A]$}\noLine
\UIC{}
 \DP
  \\
  \hline
\end{tabular}\vspace{0.3cm}
\caption{Justification axioms with corresponding
tableau rules.}\label{table:E-tableau rules JL}
\end{table}

\begin{definition}
An M-model $\M=(\E,\V)$ satisfies $[t,A]$ provided $A \in \E(t)$, and satisfies $\sim [t,A]$ provided $A \not\in \E(t)$.
\end{definition}

Soundness of $\JL_\CS^e$-tableau systems  is a consequence of the following lemma.

\begin{lemma}
Let $\pi$ be any branch of a $\JL^e_\CS$-tableau and $\M$ be a $\JL_\CS$-model that satisfies all the formulas occur in $\pi$. If a $\JL^e_\CS$-tableau rule is applied to $\pi$, then it produces at least one extension $\pi'$ such that $\M$ satisfies all the formulas occur in $\pi'$.
\end{lemma}

\begin{theorem}[Soundness]
If $A$ has a $\JL^e_\CS$-tableau proof, then it is $\JL_\CS$-valid.
\end{theorem}

In order to prove completeness we use the following cut rules

\[
\AXC{}\RightLabel{$(cut_f)$}
\UIC{$A~|~\neg A$}
\DisplayProof
\qquad
\AXC{}\RightLabel{$(cut_e)$}
\UIC{$[t,A]~|~\sim [t,A]$}
\DisplayProof
\]
Sometimes we denote the above cut rules with the following single rule
\[
\AXC{}\RightLabel{$(cut)$}
\UIC{$\varphi~|~- \varphi$}
\DisplayProof
\]
where $\varphi$ is an $\JL$-formula and $-$ is $\neg$ or it is an evidential formula and $-$ is $\sim$. Thus $(cut)$ denotes either $(cut_f)$ or $(cut_e)$.
The cut rule is the same as the rules $(PB)$ and $(PB_e)$ but without any restrictions on the cut-formula $\varphi$. Completeness is proved by first showing that all theorems of $\JL_\CS$ are provable in the tableau system $\JL_\CS^e + (cut)$, and then by proving the cut elimination theorem for $\JL_\CS^e + (cut)$.

\begin{theorem}\label{thm:completeness JL^e+cut}
If $A$ is provable in $\JL_\CS$, then it is provable in the tableau system $\JL_\CS^e + (cut)$.
\end{theorem}
\begin{proof}
The proof is by induction on the proof of $A$ in $\JL_\CS$. It is a routine matter to check that all axioms of $\JL$ are provable in $\JL_\CS^e$, even without using $(PB)$, $(PB_e)$ and $(cut)$. If $A$ is obtained from $B$ and $B\r A$ by MP, then by the induction hypothesis there are closed $\JL_\CS^e$-tableaux $T_1$ and $T_2$ for $B$ and $B\r A$ respectively. Then, using the cut rule twice, the following is a closed tableau for $A$

\vspace*{0.2cm}
\Tree[.$\neg A$ [.$B$ [.$B\r A$ {$\neg B$ \\ $\otimes$} {$A$ \\ $\otimes$} !\qsetw{2cm} ] \qroof{$T_2$}.$\neg (B\r A)$ !\qsetw{3cm} ] \qroof{$T_1$}.$\neg B$ !\qsetw{2.5cm} ]
\vspace*{0.2cm}

Finally, if $A=c:F\in\CS$ is obtained by IAN, then by the closure condition $\neg c:F$ is a closed one-node tableau. \qed
\end{proof}

The proof of the cut elimination is similar to the algorithm given by Fitting in \cite{Fitting1996}, and thus the details will be omitted. The following definitions are inspired from those in \cite{Fitting1996}.

\begin{definition}
The rank of a term $t$ and a formula $A$, denoted by $r(t)$ and $r(A)$ respectively, is defined inductively as follows:
\begin{enumerate}
\item $r(x)=r(c)=0$, for justification variable $x$ and justification constant $c$,\\
$r(s+t) = r(s\cdot t)= r(s) + r(t) +1$, $r(!t)=r(\bar{?}t)=r(?t) = r(t) +1$.

\item $r(p)=r(\bot)=0$, for $p\in\mathcal{P}$, \\ $r(\neg A)= r(A) +1$, $r(A \r B) = r(A) + r(B) +1$, $r(t:A) = r(t) + r(A) +1$.
\end{enumerate}
The rank of evidential formulas are defined as follows: $r([t,A]) := r(t) + r(A)$.
\end{definition}

\begin{definition}
Suppose that in a tableau $T$ there is a cut to $\varphi$ and $- \varphi$ of the following form:

\vspace*{0.2cm}
\Tree [ \qroof{$T_1$}.$\varphi$ \qroof{$T_2$}.$-\varphi$ !{\qbalance} ]
\vspace*{0.2cm}

where $T_1$ and $T_2$ are the subtableaux below $\varphi$ and $- \varphi$, respectively. Let $|T|$ denote the number of formulas in the tableau $T$.
\begin{enumerate}
\item We say the cut is at a branch end if $|T_1|=0$ or $|T_2|=0$; that is, if either there are no formulas below $\varphi$, or there are no formulas below $- \varphi$, or both.

\item The rank of the cut is the rank of the cut-formula $\varphi$.

\item The weight of the cut is the number of formulas in $T$ strictly below $\varphi$ and $- \varphi$; that is, the weight of the cut is $|T_1| + |T_2|$.

\item The cut is called minimal if there are no cuts in the subtableaux $T_1$ and $T_2$.
\end{enumerate}

\end{definition}

The following fact will be used frequently in the proof of cut elimination (cf. \cite{Fitting1996}). Suppose that $T$ is a closed tableau for a finite set $S$ of formulas and $S \subseteq S'$, where $S'$ is also finite. Then there is a closed tableau for $S'$ with the same number of steps.

\begin{theorem}[Cut Elimination]\label{thm:Cut Elimination}
If a formula is provable in the tableau system $\JL_\CS^e+ (cut)$, then it is also provable in $\JL_\CS^e$.
\end{theorem}

\begin{proof}
We will show how to eliminate the minimal cuts from a tableau $T$. Suppose $T$ consists a minimal cut of the following form:

\vspace*{0.2cm}
\Tree [.$\Theta$ \qroof{$T_1$}.$\varphi$ $(cut)$ \qroof{$T_2$}.$-\varphi$ !\qsetw{1cm} !{\brOverride} ]
\vspace*{0.2cm}

The proof is by induction on the rank of the cut-formula $\varphi$ with subinduction on the weight of the cut. Similar to the cut elimination of the sequent calculus of classical logic (cf. \cite{TS}), we distinguish three cases:

\begin{description}
\item[Case I.] The minimal cut is at a branch end.

\item[Case II.] The minimal cut is not at a branch end, and the uppermost formulas in $T_1$ or $T_2$ are obtained by applying a tableau rule to a formula from $\Theta$.

\item[Case III.] The minimal cut is not at a branch end, and the uppermost formulas in $T_1$ and $T_2$ are obtained by applying tableau rules to $\varphi$ and $- \varphi$, respectively.
\end{description}

In case I, we eliminate the minimal cut. In cases II and III, we transform the tableau $T$ into another closed tableau in which the minimal cut is replaced by cuts of lower rank, by cuts of the same rank but of lower weight, or both.

{\bf Case I.}
Suppose we have a minimal cut at the end of a branch. We only consider two cases   (see \cite{Fitting1996} for other cases). Consider the following cut to formulas $[t,A]$ and $\sim [t,A]$.

\vspace*{0.2cm} 
\Tree [.$\Theta$ \qroof{$T'$}.$[t,A]$ {$\sim [t,A]$ \\ $\otimes$} !{\qbalance} ]
\vspace*{0.2cm}

Suppose $\sim [t,A]$ plays a role in the closure of the right fork, otherwise the cut can be eliminated. Hence $[t,A]$ must occur in $\Theta$. Thus the cut adds a redundant formula to the left fork. Hence if the cut is eliminated, we still have closure. The case in which the left fork is closed is similar.

Let us also consider the case in which the branch closes because of $\neg c:F$, where $c:F \in \CS$. In this case the cut looks like this.

\vspace*{0.2cm}
\Tree [.$\Theta$ \qroof{$T'$}.$c:F$ {$\neg c:F$ \\ $\otimes$} !{\qbalance} ]
\vspace*{0.2cm}

Since $c:F \in \CS$, the cut-formula $c:F$ is a $\CS$-subformula of the root, and hence the cut is an instance of $(PB)$.

{\bf Case II.}
Suppose the minimal cut is not at a branch end, and the uppermost formulas in $T_1$ or $T_2$ are obtained by applying a tableau rule to formulas from $\Theta$. In this case we push the cut down in the tableau and obtain a new cut of lower weight. We only consider two cases: (i) the rule $(\cdot)$ is applied to formulas from $\Theta$, and (ii) the rule $(PB)$ is applied. Other cases are similar.

Suppose the rule $(\cdot)$ is applied to formulas from $\Theta$. Thus the cut is of the form shown in (1), where $A \r B$ is a $\CS$-subformula of the root and $s \cdot t$ occurs in the root. The displayed cut in (1) is transformed into the one in (2) of lower weight.

(1)
\vspace*{0.2cm}
\Tree [.{$\vdots$ \\ $[s,A \r B]$ \\ $[t,A]$ \\ $\vdots$} \qroof{$T_1$}.{$\varphi$ \\ $[s \cdot t,B]$} \qroof{$T_2$}.$-\varphi$ !{\qbalance} ]
\hskip 1.5cm
(2)
\Tree [.{$\vdots$ \\ $[s,A \r B]$ \\ $[t,A]$ \\ $\vdots$ \\$[s \cdot t,B]$} \qroof{$T_1$}.$\varphi$ \qroof{$T_2$}.$-\varphi$ !{\qbalance} ]
\vspace*{0.1cm}

Now suppose the rule $(PB)$ is applied. Then the cut is of the form shown in (3), where $A$ is a $\CS$-subformula of the root. The displayed cut in (3) is transformed into the cuts in (4) of lower weights.

\vspace*{0.2cm}
(3)
\renewcommand{\qroofpadding}{0.6em}
\Tree [.$\Theta$
[.$\varphi$ \qroof{$T_1^L$}.$A$ !\qsetw{1cm} $(PB)$ \qroof{$T_1^R$}.$\neg A$ !{\brOverride} ].$\varphi$ !{\brRestore} !\qsetw{4cm}
\qroof{$T_2$}.$-\varphi$ ] 
\hskip 1.5cm
(4)
\Tree [.$\Theta$
[.$A$ \qroof{$T_1^L$}.$\varphi$ !\qsetw{1.5cm} \qroof{$T_2$}.$-\varphi$ ] 
$(PB)$
[.$\neg A$ \qroof{$T_1^R$}.$\varphi$ !\qsetw{1.5cm} \qroof{$T_2$}.$-\varphi$ ] !{\brOverride} ]
\vspace*{0.1cm}

{\bf Case III.}
Suppose the minimal cut is not at a branch end, and the uppermost formulas in $T_1$ and $T_2$ are obtained by applying tableau rules to $\varphi$ and $- \varphi$, respectively. In this case we transform the cut into cuts of lower rank, or into cuts with the same rank but of lower weight. 

First consider the rule $(\cdot)$ is applied to the cut formula.

\vspace*{0.2cm}
(5)
\Tree [.{$\Theta$ \\ $[s,A\r B]$} 
[.$[t,A]$ \qroof{$T_1$}.$[s\cdot t,B]$ ]
$(PB_e)$
\qroof{$T_2$}.$\sim[t,A]$  !{\qbalance} !{\brOverride} ].{$\Theta$ \\ $[s,A\r B]$}
\hskip 0.5cm
(6)
\Tree [.{$\Theta$ \\ $[t,A]$} 
[.$[s,A\r B]$ \qroof{$T_1$}.$[s\cdot t,B]$ ]
$(PB_e)$
\qroof{$T_2$}.$\sim[s,A\r B]$ !{\qbalance} !{\brOverride} ].{$\Theta$ \\ $[t,A]$}
\vspace*{0.2cm}

Since $A \r B$ is a $\CS$-subformula of the root and $s \cdot t$ occurs in the root, the two cuts shown in (5) and (6) are instances of $(PB_e)$.
For example, the following cut is an instance of $(PB_e)$.

\vspace*{0.2cm}
\Tree [.{$\Theta$ \\ $[s',A\r B]$}
[.$[t+s,A]$ \qroof{$T_1$}.$[s'\cdot(t+s),B]$ ]
$(PB_e)$
[.$\sim[t+s,A]$ \qroof{$T_2$}.$\sim[t,A]$ ] !{\qbalance} !{\brOverride} ].{$\Theta$ \\ $[s',A\r B]$}
\vspace*{0.2cm}

Now consider the following cut to formulas $t:A$ and $\neg t:A$ to which the rules $(Te)$ and $(Fe)$ are applied respectively.

\vspace*{0.2cm}
\Tree [.$\Theta$
[.$t:A$ \qroof{$T_1$}.$[t,A]$ ]
$(cut)$
[.$\neg t:A$ \qroof{$T_2$}.$\sim [t,A]$ ] !{\qbalance} !{\brOverride} ].$\Theta$
\vspace*{0.2cm}

This cut is transformed into the following cuts.

\vspace*{0.2cm}
\Tree [.$\Theta$
[.$[t,A]$ \qroof{$T_1$}.$t:A$ $(cut)_2$ !\qsetw{2cm}
 [.$\neg t:A$ {$\sim [t,A]$ \\ $\otimes$} ] !{\brOverride} ].$[t,A]$ !{\brRestore}
$(cut)_1$
[.$\sim [t,A]$  [.$t:A$ {$[t,A]$ \\ $\otimes$} ] $(cut)_3$ !\qsetw{2cm} \qroof{$T_1$}.$\neg t:A$ !{\brOverride} ].$\sim [t,A]$  !{\brRestore} !{\qbalance} !{\brOverride} ].$\Theta$

The rank of $(cut)_1$ is less than the rank of $(cut)$. Moreover, $(cut)_2$ and $(cut)_3$ have the same rank as $(cut)$ but their weights are smaller than the weight of $(cut)$.

Consider the following cut to formulas $\neg t:A$ and $\neg\neg t:A$ to which the rules $(Fe)$ and $(F\neg)$ are applied respectively.

\vspace*{0.2cm}
\Tree [.$\Theta$
 [.$\neg t:A$ \qroof{$T_1$}.$\sim [t,A]$ ]
 $(cut)$
 [.$\neg\neg t:A$ \qroof{$T_2$}.$t:A$ ] !{\qbalance}  !{\brOverride} ].$\Theta$
\vspace*{0.2cm}

This cut is transformed into the following cuts.

\vspace*{0.2cm}
\Tree [.$\Theta$
  [.$t:A$ 
  [.$[t,A]$ {$\neg t:A$  \\ $\otimes$}   $(cut)_4$ !\qsetw{1.2cm} \qroof{$T_2$}.$\neg\neg t:A$ !{\brOverride} ].$[t,A]$ !{\brRestore} !\qsetw{2.5cm}
   $(cut)_2$
     [.$\sim [t,A]$ {$[t,A]$ \\ $\otimes$} ] !{\brOverride} ].$t:A$ !{\brRestore} !\qsetw{0.7cm}
    $(cut)_1$ 
    !\qsetw{0.7cm}
   [.$\neg t:A$ [.$[t,A]$  {$\sim[t,A]$ \\ $\otimes$} ]   !\qsetw{2.5cm}
    $(cut)_3$
     \qroof{$T_1$}.$\sim[t,A]$ !{\brOverride} ].$\neg t:A$ !{\brRestore} !{\brOverride}
  ].$\Theta$ 
\vspace*{0.2cm}

The ranks of $(cut)_1$, $(cut)_2$,  and $(cut)_3$ are less than the rank of $(cut)$. Moreover, $(cut)_4$ has the same rank as $(cut)$ but its weight is smaller than the weight of $(cut)$.

Now suppose that jT is an axiom of $\JL$. Consider the following cut to formulas $[t+s,A]$ and $\sim [t+s,A]$ to which the rules $(e)$ and $(+_L)$ are applied respectively.

\vspace*{0.2cm}
\Tree [.$\Theta$
[.$[t+s,A]$ \qroof{$T_1$}.$A$ ]
$(cut)$
[.$\sim [t+s,A]$ \qroof{$T_2$}.$\sim [t,A]$ ] !{\qbalance} !{\brOverride} ].$\Theta$
\vspace*{0.2cm}

This cut is transformed into the following cuts.

\vspace*{0.2cm}
\Tree [.$\Theta$
[.$[t,A]$ [.$A$ \qroof{$T_1$}.$[t+s,A]$ 
$(cut)_4$
 [.$\sim [t+s,A]$ {$\sim [t,A]$ \\ $\otimes$} ] !{\brOverride} ].$A$ !{\brRestore} !\qsetw{2.5cm}
$(cut)_2$
[.$\neg A$ {$A$ \\ $\otimes$} ] !{\brOverride} ].$[t,A]$ !\qsetw{4.5cm}
$(cut)_1$
[.$\sim [t,A]$ [.$[t+s,A]$ \qroof{$T_1$}.$A$ !{\brRestore} ].$[t+s,A]$ !\qsetw{2.5cm}
$(cut)_3$
\qroof{$T_2$}.$\sim [t+s,A]$ !{\brOverride} ].$\sim [t,A]$ !{\brOverride} ].$\Theta$
\vspace*{0.2cm}

The ranks of $(cut)_1$ and $(cut)_2$ are less than the rank of $(cut)$. Moreover, $(cut)_3$ and $(cut)_4$ have the same rank as $(cut)$ but their weights are smaller than the weight of $(cut)$. The case in which rule $(+_R)$ is applied instead of rule $(+_L)$ is similar.

Consider the following cut to formulas $[!t,t:A]$ and $\sim [!t,t:A]$ to which the rules $(e)$ and $(!)$ are applied respectively.

\vspace*{0.2cm}
\Tree [.$\Theta$
[.$[!t,t:A]$ \qroof{$T_1$}.$t:A$ ]
$(cut)$
[.$\sim [!t,t:A]$ \qroof{$T_2$}.$\sim [t,A]$ ] !{\qbalance} !{\brOverride} ].$\Theta$
\vspace*{0.2cm}

This cut is transformed into the following cuts.

\vspace*{0.2cm}
\Tree [.$\Theta$
[.$[t,A]$ 
[.$t:A$ 
\qroof{$T_1$}.$[!t,t:A]$
$(cut)_4$
[.$\sim [!t,t:A]$ {$\sim[t,A]$ \\ $\otimes$} ]
!{\brOverride} ].$t:A$ !{\brRestore}
!\qsetw{2cm}
$(cut)_2$
[.$\neg t:A$ {$\sim [t,A]$ \\ $\otimes$} ]
!{\brOverride} ].$[t,A]$ !{\brRestore}
!\qsetw{0.4cm}
$(cut)_1$
!\qsetw{0.4cm}
[.$\sim [t,A]$ 
[.$t:A$ {$[t,A]$ \\ $\otimes$} ].$t:A$
!\qsetw{0.5cm}
$(cut)_3$
!\qsetw{0.5cm}
[.$\neg t:A$
[.$[!t,t:A]$ {$t:A$ \\ $\otimes$} ]
$(cut)_5$
\qroof{$T_2$}.$\sim [!t,t:A]$
!{\brOverride} ].$\neg t:A$
!{\brOverride} ].$\sim [t,A]$ !{\brRestore}
!{\brOverride} ].$\Theta$  
\vspace*{0.2cm}

The ranks of $(cut)_1$, $(cut)_2$, and $(cut)_3$ are less than the rank of $(cut)$. Moreover, $(cut)_4$ and $(cut)_5$ have the same rank as $(cut)$ but their weights are smaller than the weight of $(cut)$.
The cut to formulas $[?t,\neg t:A]$ and $\sim [?t,\neg t:A]$ to which the rules $(e)$ and $(?)$ are applied respectively is treated similarly.

Consider the following cut to formulas $[\bar{?}t,\neg t:A]$ and $\sim [\bar{?}t,\neg t:A]$ to which the rules $(e)$ and $(\bar{?})$ are applied respectively.

\vspace*{0.2cm}
\Tree [.$\Theta$
[.$[\bar{?}t,\neg t:A]$ \qroof{$T_1$}.$\neg t:A$ ]
$(cut)$
[.$\sim [\bar{?}t,\neg t:A]$ \qroof{$T_2$}.$A$ ] !{\qbalance} !{\brOverride} ].$\Theta$
\vspace*{0.2cm}

This cut is transformed into the following cuts.

\vspace*{0.2cm}
\Tree [.$\Theta$
[.$t:A$ 
[.$A$ 
[.$[\bar{?}t,\neg t:A]$ {$\neg t:A$ \\ $\otimes$} ] !\qsetw{1.8cm} 
$(cut)_4$ !\qsetw{1.2cm} 
\qroof{$T_2$}.$\sim [\bar{?}t,\neg t:A]$ !{\brOverride} ].$A$ !{\brRestore} !\qsetw{3.1cm} 
$(cut)_2$ !\qsetw{0.1cm} 
[.$\neg A$ {$[t,A]$ \\ $A$ \\ $\otimes$} ] !{\brOverride} ].$t:A$ 
$(cut)_1$ 
\qroof{$T_3$}.$\neg t:A$ 
 !{\brOverride} ].$\Theta$
\vspace*{0.2cm}

In which the subtableau $T_3$ is as follows.

\Tree[.$\neg t:A$ 
[.$A$  
\qroof{$T_1$}.$[\bar{?}t,\neg t:A]$  !\qsetw{1.8cm} 
$(cut)_5$ !\qsetw{1.2cm} 
\qroof{$T_2$}.$\sim [\bar{?}t,\neg t:A]$ 
!{\brOverride} ].$A$ !{\brRestore}
$(cut)_3$
[.$\neg A$  
\qroof{$T_1$}.$[\bar{?}t,\neg t:A]$ !\qsetw{1.8cm} 
$(cut)_6$ !\qsetw{1.2cm} 
[.$\sim [\bar{?}t,\neg t:A]$ {$A$ \\ $\otimes$} ] 
!{\brOverride} ].$\neg A$ !{\brRestore}
!{\brOverride} ].$\neg t:A$

The ranks of $(cut)_1$, $(cut)_2$,  and $(cut)_3$ are less than the rank of $(cut)$. Moreover, $(cut)_4$, $(cut)_5$ and $(cut)_6$ have the same rank as $(cut)$ but their weights are smaller than the weight of $(cut)$.

Consider the cut to formulas $[t+s,\bot]$ and $\sim [t+s,\bot]$ shown in (7) to which the rules $(e_\bot)$ and $(+_L)$ are applied respectively.

\vspace*{0.2cm}
(7)
\Tree [.$\Theta$ [.$[t+s,\bot]$ {$\bot$ \\ $\otimes$} ] [.$\sim [t+s,\bot]$ \qroof{$T'$}.$\sim [t,\bot]$ ] !{\qbalance} ]
\hskip 1.5cm
(8)
\Tree [.$\Theta$ [.$[t,\bot]$ {$\bot$ \\ $\otimes$} ]
[.$\sim [t,\bot]$ [.$[t+s,\bot]$ {$\bot$ \\ $\otimes$} ] \qroof{$T'$}.$\sim [t+s,\bot]$ !\qsetw{2.5cm} ] !\qsetw{4.5cm} ]
\vspace*{0.2cm}

The cut in (7) is transformed into the cuts shown in (8), in which the cut to $[t,\bot]$ and $\sim [t,\bot]$ has a lower rank, and the weight of the cut to $[t+s,\bot]$ and $\sim [t+s,\bot]$ is smaller than the weight of the original cut in (7). The case in which  $(+_R)$ is applied instead of $(+_L)$ is treated in a similar way.

Actually there are two remaining cuts to verify in  case III: the cut to formulas $A \r B$ and $\neg (A \r B)$ to which the rules $(T\r)$ and $(F\r)$ are applied respectively, and the cut to formulas $\neg (A \r B)$ and $\neg\neg (A \r B)$ to which the rules $(F\r)$ and $(F\neg)$ are applied respectively. We refer the reader to \cite{Fitting1996} for a more detailed exposition of these two cuts. \qed
\end{proof}

\begin{theorem}[Completeness]\label{thm:completeness E-tableaux}
If $A$ is $\JL_\CS$-valid, then it has a $\JL^e_\CS$-tableau proof.
\end{theorem}
\begin{proof}
If $A$ is $\JL_\CS$-valid, then by Theorem \ref{thm:Sound Compl JL} it is provable in $\JL_\CS$. Hence, by Theorem \ref{thm:completeness JL^e+cut}, it is provable in $\JL^e_\CS + (cut)$. Then, by the cut elimination theorem, it is provable in $\JL^e_\CS$. \qed
\end{proof}

Inspection of all $\JL^e_\CS$-tableau rules of Tables \ref{table:E-J} and \ref{table:E-tableau rules JL}   shows that the following subformula property holds.

\begin{theorem}[Subformula property]\label{thm:subformula property E-tableaux}
Every $\JL$-formula in a $\JL^e_\CS$-tableau proof is a weak $\CS$-subformula of the root of the tableau. Every $\JL$-formula in an evidential formula is a  $\CS$-subformula of the root of the tableau, and every term in an evidential formula occurs in the root or in a formula in $\CS$.
\end{theorem}

Note that for a finite constant specification $\CS$, the set of all $\CS$-subformulas of a formula is finite. Thus, in order to prove a given formula we need to search for a finite number of formulas. Furthermore, the number of applications of each tableau rule is finite too. In fact, the only tableau rules that increase the complexity of formulas or terms are $(PB)$, $(PB_e)$ and $(\cdot)$. For a finite constant specification, the set of all $\CS$-subformulas of a formula and the set of all terms occur in it are finite, and hence it is enough to apply $(PB)$ and $(PB_e)$ only finitely many times. Moreover,  the condition given in Table \ref{table:E-J} for the rule $(\cdot)$ avoids infinite applications of this rule. Thus we have

\begin{theorem}
Given any finite constant specification $\CS$ for $\JL$, ${\sf JL}^e_\CS$-tableaux always terminate. Therefore, justification logic $\JL_\CS$ is decidable.
\end{theorem}

Finally we present $\JL_\CS^e$-tableaux with signed formulas. In this case tableaux are constructed from singed formulas of the following forms:
\[
T~A,\quad F~A, \quad T~[t,A], \quad F~[t,A],
\]
where $A$ is a $\JL$-formula and $t$ is a term. Note that using singed formulas there is no need to  use the negation symbol $\sim$. All the results of this section can be adapted to $\JL_\CS^e$-tableaux with signed formulas.

\begin{table} 
\centering\renewcommand{\arraystretch}{1.5}
\begin{tabular}{|cc|}
\hline
\AXC{}\noLine 
\UIC{$F~\neg A$}\RightLabel{$(F\neg)$}
\UIC{$T~A$}\noLine
\UIC{}
\DP
&
\AXC{}\noLine 
\UIC{$T~\neg A$}\RightLabel{$(T\neg)$} 
\UIC{$F~A$}\noLine
\UIC{}
\DP
\\
\AXC{}\noLine 
\UIC{$F~A\rightarrow B$}\RightLabel{$(F\r)$}
\UIC{$T~A$}\noLine
\UIC{$F~B$}\noLine
\UIC{}
\DisplayProof
&
\AXC{$T A\rightarrow B$}\RightLabel{$(T\r)$}
\UIC{$F~A |T~B$}
\DP
\\\hline

\AXC{}\noLine\UIC{$T~t:A$}\RightLabel{$(Te)$} 
\UIC{$T~[t,A]$}\noLine
\UIC{}
\DP
&
\AXC{}\noLine\UIC{$F~t:A$}\RightLabel{$(Fe)$} 
\UIC{$F~[t,A]$}\noLine
\UIC{}
\DP
\\\hline
\AXC{}\noLine 
\UIC{$F~[t+s,A]$}\RightLabel{$(+_L)$}
\UIC{$F~[t,A]$}
\DisplayProof
&
\AXC{}\noLine 
\UIC{$F~[t+s,A]$}\RightLabel{$(+_R)$}
\UIC{$F~[s,A]$}
\DisplayProof
\\
\multicolumn{2}{|c|}{
\AXC{}\noLine
\UIC{}\noLine
\UIC{$T~[s,(A\rightarrow B)]$}\noLine
\UIC{$T~[t,A]$}\RightLabel{$(\cdot)$}
\UIC{$T~[s\cdot t,B]$}\noLine
\UIC{}
\DisplayProof
}
\\\hline
\AXC{}\noLine 
\UIC{}\noLine 
\UIC{} 
\RightLabel{$(PB)$}
\UIC{$T~A~|F~A$}\noLine
\UIC{}
\DisplayProof 
&
\AXC{}\noLine 
\UIC{}\noLine 
\UIC{} 
\RightLabel{$(PB_e)$}
\UIC{$T~[t,A]~|F~[t,A]$}\noLine
\UIC{}
\DisplayProof
\\\hline
\multicolumn{2}{|l|}{In $(\cdot)$ the formula $A \r B$ is a $\CS$-subformula of the root and}\\\multicolumn{2}{|l|}{ the term $s \cdot t$ occurs in the root.}\\\hline
\multicolumn{2}{|l|}{In $(PB)$ the $PB$-formula $A$ is a $\CS$-subformula of the root.}\\\hline
\multicolumn{2}{|l|}{In $(PB_e)$ the formula $A$ is a $\CS$-subformula of the root and}\\\multicolumn{2}{|l|}{ the term $t$ occurs in the root.}\\\hline
\end{tabular}\vspace{0.3cm}
\caption{Signed tableau rules for basic justification logic {\sf J}.}\label{table:Singed-E-J}
\end{table}

\begin{table}
\centering\renewcommand{\arraystretch}{2}
\begin{tabular}{|l|c|}
\hline
~Justification axiom & Tableau rule~ \\\hline
~{\bf jT}. $t:A\r A$ &
\AXC{}\noLine 
\UIC{$T~[t,A]$} \RightLabel{$(e)$}
\UIC{$T~A$}\noLine
\UIC{}
\DP \\\hline
~{\bf jD}. $t:\bot\r\bot$ &
\AXC{}\noLine 
\UIC{$T~[t,\bot]$} \RightLabel{$(e_\bot)$}
\UIC{$T~\bot$}\noLine
\UIC{}
\DP
\\\hline
~{\bf j4}. $t:A\r !t:t:A$ &
\AXC{}\noLine 
\UIC{$F~[!t,t:A]$} \RightLabel{$(!)$}
\UIC{$F~[t,A]$}\noLine
\UIC{}
\DP
\\\hline
~{\bf jB}. $\neg A\r \bar{?}t:\neg t:A$ &
\AXC{}\noLine 
\UIC{$F~[\bar{?}t,\neg t:A]$} \RightLabel{$(\bar{?})$}
\UIC{$T~A$}\noLine
\UIC{}
\DP
\\\hline
~{\bf j5}. $\neg t:A\r ?t:\neg t:A$ &
\AXC{}\noLine 
\UIC{$F~[?t,\neg t:A]$} \RightLabel{$(?)$}
\UIC{$T~[t,A]$}\noLine
\UIC{}
\DP
\\
\hline
\end{tabular}\vspace{0.3cm}
\caption{Justification axioms with corresponding signed
tableau rules.}\label{table:Signed-E-tableau rules JL}
\end{table}

\noindent
{\bf Acknowledgments}\\

This research was in part supported by a grant from IPM. (No. 95030416)



\begin{thebibliography}{99}

\bibitem{A1995} S. Artemov, Operational modal logic, Technical Report MSI 95--29, Cornell
University, 1995.

\bibitem{A2001} S. Artemov, Explicit provability and constructive semantics, The Bulletin
of Symbolic Logic, 7(1), 1--36, 2001.


\bibitem{A2008} S. Artemov, The logic of justification, The Review of Symbolic Logic, 1(4), 477--513, 2008.

\bibitem{ArtemovFitting} S. Artemov and M. Fitting, Justification logic, In Edward N. Zalta, editor, The Stanford Encyclopedia of Philosophy, 2012.

\bibitem{D’Agostino1992} M. D'Agostino, Are tableaux an improvement on truth-tables?— Cut-free proofs and bivalence,
Journal of Logic, Language and Information, 1, 235--252, 1992.


\bibitem{D’Agostino1999} M. D'Agostino, Tableau methods for classical propositional logic, In Handbook of Tableau
Methods, M. D’Agostino, D. Gabbay, R. Haehnle, and J. Posegga, eds, pp. 45--124. Kluwer,  1999.

\bibitem{D’AgostinoMondadori1994} M. D'Agostino and M. Mondadori, The taming of the cut. Classical refutations with analytic cut, Journal of Logic and Computation, 4, 285--319, 1994.

\bibitem{Finger2010} M. Finger, Analytic methods for the logic of proofs. Journal of Logic and Computation, 20(1), 167--188,  2010.

\bibitem{Fitting1996} M. Fitting, First-Order Logic and Automated Theorem Proving, Springer, 1996, Second Edition.

\bibitem{Fitting2005} M. Fitting, The logic of proofs, semantically, Annals of Pure and Applied Logic, 132(1), 1--25, 2005.


\bibitem{Fitting2008} M. Fitting, Reasoning with justifications, In David Makinson, Jacek Malinowski, and Heinrich Wansing, editors, Towards Mathematical Philosophy, Papers from the Studia Logica conference Trends in Logic IV, volume 28 of Trends in Logic, chapter 6, pages 107--123. Springer, 2009.


\bibitem{Ghari-tableaux-2016}  M. Ghari,
Tableau proof systems for justification logics, ArXiv e-prints, arXiv:1405.1828v5, April 2016.

\bibitem{KuznetsPhD2008} R. Kuznets, Complexity Issues in Justification Logic, PhD thesis, City University of New York, May 2008.

\bibitem{KuznetsStuder2012} R. Kuznets and T. Studer, Justifications, ontology, and conservativity, In Thomas Bolander, Torben Braüner, Silvio Ghilardi, and Lawrence Moss, editors, Advances in Modal Logic, Volume 9, pages 437--458. College Publications, 2012.

\bibitem{Mkrtychev1997} A. Mkrtychev, Models for the logic of
proofs, In S. I. Adian, A. Nerode (Eds.), Logical Foundations of
Computer Science, Vol. 1234 of Lecture Notes in Computer Science,
Springer, 1997, pages 266--275.


\bibitem{Renne2004} B. Renne, Tableaux for the Logic of Proofs, Technical Report TR--2004001, CUNY Ph.D. Program in Computer Science, March 2004.

\bibitem{Renne2006} B. Renne, Semantic cut-elimination for two explicit modal logics, In Janneke Huitink and Sophia Katrenko, editors, Proceedings of the Eleventh ESSLLI Student Session, 18th European Summer School in Logic, Language and Information (ESSLLI'06), pages 148--158, 2006.


\bibitem{TS} A. Troelstra and H. Schwichtenberg, Basic Proof Theory,Cambridge University Press, Amsterdam, 1996.

\end{thebibliography}
\end{document}